\documentclass[12pt]{amsart}

\title[Large Level Sets]{Dimension of Images of Large Level Sets}

\author{Gavin Armstrong}
\address[GA]{ Mathematics Department,
Central Washington University,
400 E University Way,
Ellensburg,
WA 98926, USA}
\email[GA]{Gavin.Armstrong@cwu.edu}

\author{Anthony G. O'Farrell}
\address[AOF]{Mathematics and Statistics, Maynooth University,
Co Kildare, W23 HW31, Ireland}
\email[AOF]{anthony.ofarrell@mu.ie}

\RequirePackage{datetime, graphicx, url}
\date{\today\ \hhmmsstime}

\usepackage{tikz}
\usetikzlibrary{decorations.pathreplacing}
\usetikzlibrary{arrows}

\newcommand\ignore[1]{}

\newcommand{\dist}{\hbox{dist}}

\newcommand{\R}{\mathbb{R}}

\newcommand{\N}{\mathbb{N}}

\RequirePackage{mathrsfs}
\newcommand{\Hs}{\mathscr{H}}

\newcommand{\diam}{\textup{diam}}
\newcommand{\Lip}{\textup{Lip}}
\newcommand{\preimy}{f^{-1}(y)}

\theoremstyle{plain} 
\newtheorem{theorem}{Theorem}
\newtheorem{lemma}{Lemma}

\theoremstyle{definition}


\begin{document}

\begin{abstract}
Let $k$ be a natural number.
We consider $k$-times continuously-differentiable 
real-valued functions
$f:E\to\R$, where $E$ is some interval on the line
having positive length.  For $0<\alpha<1$ let 
$I_\alpha(f)$ denote the set of values $y\in\R$
whose preimage $f^{-1}(y)$ has Hausdorff dimension
$\dim f^{-1}(y) \ge \alpha$.  We consider how large can be
the Hausdorff dimension of $I_\alpha(f)$,
as $f$ ranges over the set $C^k(E,\R)$
of all $k$-times continuously-differentiable functions
from $E$ into $\R$.  We show that the sharp upper bound
on $\dim I_\alpha(f)$ is $\displaystyle\frac{1-\alpha}k$.
\end{abstract}

\maketitle

\section{Introduction}\label{S:1}

\subsection{Big level sets}

Let $f$ be a continuously-differentiable real-value function
of a real variable defined on some interval $I\subset\R$,
and consider its level sets $f^{-1}(y)$, for $y\in\R$.

In everyday experience of nice smooth functions
mapping $\R\to\R$, one
encounters only level sets that are discrete. 
But pathology
is not far below the surface.  For instance,
for an arbitrary closed set $F\subset\R$,
the function $f:x\mapsto \dist(x,F)^{k+1}$ 
is $k$-times continuously-differentiable on $\R$
and has $f^{-1}(0)=F$.  It is not difficult to
construct examples for which infinitely many
level sets are as large and wild as one might wish.       

For how many $y$ could 
$f^{-1}(y)$ be `large'?  This depends on the meaning of
`large'.  Clearly, at most a countable number of level sets may
have positive length (i.e. positive one-dimensional Lebesgue
measure).  Simple examples such as the sine function
show that all the nonempty level sets
may be infinite.  What about interpretations of `large'
intermediate between `infinite' and of `positive length'?

\subsection{Uncountable level sets}\label{SS:1-1}
Consider the set
$$ I_u(f):= \{ y\in\R: f^{-1}(y) \textup{ is uncountable}\}.
$$ 
It may certainly happen that $I_u(f)$ is
itself uncountable. For instance, the first coordinate
of the map constructed in \cite{OF} is a $C^\infty$
function and has an uncountable number of
uncountable level sets. 
For $y\in I_u(f)$, the closed set $\preimy$ 
must contain accumulation points, and these points are
critical points of $f$, so by a well-known result of
 Morse and Sard 
\cite{Morse}, $I_u(f)$
has length zero.  In fact, for small $\delta>0$,
the intersection of 
$I_u(f)$ with any bounded interval
may be covered by O$(1/\delta)$ intervals
of length o$(\delta)$. If $f$ is smoother, say $k$ times
continuously-differentiable, then $f$ is `flat to order
$k$' (i.e. all derivatives up to order $k$ are zero)
at each accumulation point of $\preimy$ for
$y\in I_u(f)$, and it follows easily that
$I_u(f)$ has Hausdorff dimension at most
$1/k$. For background on Hausdorff measures and dimension,
and fractals, see 
\cite{Rogers, Mattila, Federer, Falconer, Bishop+Peres,
Mattila+Mauldin}. 

\subsection{Level sets of positive dimension}
To explore beyond the merely uncountable, let us
consider critical image sets of positive Hausdorff dimensions.
We denote by $\dim S$ the Hausdorff dimension of a set
$S\subset\R$.  For $0<\alpha\le1$, consider the set
$$ I_\alpha(f):= \{ y\in\R: \dim f^{-1}(y) \ge \alpha\}. $$ 
How large could $I_\alpha(f)$ be?

The most natural measure of its size is
its Hausdorff dimension:
$$ L(\alpha, f) := \dim I_\alpha(f). $$ 
This is some number in the interval $[0,1]$.
It is nonincreasing as $\alpha$ increases:
\begin{equation}
\label{E:1}
 \alpha_1 < \alpha_2 \implies L(\alpha_1,f) \ge L(\alpha_2,f). 
\end{equation}

We use $C^k$ to stand for `$k$ times continuously differentiable',
and $C^k(E,\R)$ to denote the set of all $C^k$
functions from an interval $E$ into $\R$.

Our main result is the following:

\begin{theorem}\label{T:1} Let $k\in \N$ and $0<\alpha<1$,
and let $E$ be a nonempty open interval on $\R$.
Then 
$$ \sup\left\{ L(\alpha,f): f\in C^k(E,\R) \right\} 
= \frac{1-\alpha}{k}. $$ 
\end{theorem}

Observe that it suffices to prove this result for
any one particular nonempty open interval $E$. We shall
prove it for $E=(-2,2)$.

\subsection{Structure of the paper}
The rest of the paper is organised as follows.
In Section \ref{S:2} we show that
\begin{equation}\label{E:2}
 \sup\left\{ L(\alpha,f): f\in C^k(E,\R) \right\} 
\le \frac{1-\alpha}{k}, 
\end{equation} 
 and in Section \ref{S:3} we show that
\begin{equation}\label{E:3}
 \sup\left\{ L(\alpha,f): f\in C^k(E,\R) \right\} 
\ge \frac{1-\alpha}{k}. 
\end{equation} 

\subsection{Wider Context}
There has been substantial interest in the
possible pathology of level sets of continuous functions.
These have been studied for functions mapping
$\R^m\to\R^n$ and more generally for functions
on and to more general topological spaces.  

Notorious examples of \emph{worst-case behaviour} have
been known for a long time. For example, the
components $f_j$ ($j=1,2$) of a Peano space-filling curve
$f=(f_1,f_2):\R\to\R^2$
obviously have an interval of values $y$ each
having a nonempty perfect set contained in $f^{-1}(y)$.
The existence of such curves in suitable H\"older classes
has been explored \cite{Buckley}.
Components of H\"older-continuous arcs passing through
product Cantor sets (similar to the example in
\cite{Garnett}) can achieve level sets of any
dimension less than 1 whose combined image has
any dimension less than 1. Further the space
$C([0,1],\R)$ contains a dense (residual) set of
functions for each of which all but a countable number of
the nonempty level sets are perfect
\cite[Theorem 2]{Garg}.  

Generically, functions $f\in C([0,1],\R)$ have level sets
of Hausdorff dimension zero \cite{Kirchheim1995}.
The survey \cite{Balka2016B} 
summarises and extends work on \emph{generic behaviour},
in various senses (Baire category, prevalence) with respect to
the usual topology on various spaces of continuous functions
$C(X,Y)$,
and focussing on the Hausdorff dimension (and other
types of dimension) of level sets. This line of
investigation goes back to Bruckner and Garg
\cite{Garg, Bruckner+Garg, Bruckner, Bruckner+Petrushka, Garg}.
See also \cite{Balka2012, Balka2014, Balka2016A, Balka2017,
Kirchheim1993, Kirchheim1992, Kirchheim1995,
Elekes, Dougherty}.
This
activity forms part of a wider theme concerning
the pathology of image sets, graphs of functions,
and slices.
See in addition \cite{Mauldin+Williams, Orponen}.  

Special kinds of functions, such as the Hardy-Weierstrass
nowhere-differentiable trigonometric series have been
studied intensively \cite{Falconer, Hunt, Shen, Buckley}.
Beyer \cite{Beyer} refined results of Kaczmarz and Steinhaus and
established facts about the
Hausdorff dimension of level sets of certain Rademacher series.
Bertoin \cite{Bertoin} computed the Hausdorff dimension of
level sets of a class of self-affine functions, building on
work of Kono on occupation densities \cite{Kono1986, Kono1988}.

Continuously-differentiable and highly-differentiable
functions have attracted less attention from the
point of view of pathological level sets, although the existence
of pathological non-generic behavior is notorious from
other aspects, such as $C^\infty$ dimension \cite{OF}, 
the structure
of critical sets \cite{Allan},
dynamics (iteration) 
and conjugation and centralisers in the case
of invertible maps, even in dimension one \cite{Roginskaya}. 
This behaviour
places formidable obstacles in the way of those
who attempt to solve open problems about general   
smooth functions. The many large level sets that appear
in our main result are not found for generic 
functions in $C^k(E,\R)$, which of course have zero-dimensional
level sets. The example functions
we construct in this paper become completely tame, with discrete level sets,
if we add a small linear perturbation, replacing
$f(x)$ by $f(x)+\epsilon x$.
 
The phenomenon exposed in the present paper may
prompt others to investigate questions about the
pathological behaviour of level sets of smooth functions
in other dimensions, and using other senses of
`large', (both for the level sets themselves and for the
image of the union).  The work mentioned above
about continuous functions involved interpretations of
`large' that are topological (such as Baire category),
metric (such as various kinds of Hausdorff dimension)
and measure-theoretic (related to Haar measure).
This could be a large investigation.
We content ourselves with the present result,
which just uses classical concepts in the simplest
context. Sufficient unto the day is the generality thereof.

\section{Upper Bound for $L(\alpha,f)$}
\label{S:2}
\subsection{}
To prove \eqref{E:2}, we have
to prove the following:

\begin{lemma} For $k\in\N$, $0<\alpha<1$ and
$f\in C^k(E,\R)$, we have
$$ L(\alpha,f) \le \frac{1-\alpha}k. $$
\end{lemma}

\begin{proof}
Suppose this lemma is false.
Then we can choose some $k\in\N$, $0<\alpha<1$,
 and $f\in C^k(E,\R)$ such that
$$ L(\alpha,f) > \frac{1-\alpha}k. $$
Fix a number $\beta$ strictly between
$L(\alpha,f)$ and 
$\frac{1-\alpha}k$. Then
$\Hs^\beta( I_\alpha(f) )$ is positive,
where $\Hs^\beta$ denotes Hausdorff measure
of dimension $\beta$.  
By a theorem
of Davies \cite[2.10.47]{Federer} (see also \cite{Howroyd}), 
we may choose a compact 
subset $Y\subset I_\alpha(f)$ having
finite positive $\Hs^\beta(Y)$. 

We have $\alpha+k\beta>1$, so we can choose
a positive number $\gamma<\alpha$ such that
$\gamma+k\beta>1$.  Then for each
$y\in Y$, the preimage $f^{-1}(y)$
has positive $\Hs^\gamma$ measure
(infinite, in fact).

Let $T$ denote the set of condensation points
of the critical set $(f')^{-1}(0)$ of $f$. 
Then $T$ is a closed subset of the interval $E$,
and
since $f$ is
flat of order $k$ at each point of $T$, we have
that 
$$ |f(x_1)-f(x_2)| = \textup{o}\left(|x_1-x_2|^k\right), (x_1,x_2\in T).$$
For each $y\in Y$, the set 
$\preimy\setminus T$ of non-condensation points 
of its preimage is countable, so
$$ \Hs^\gamma(\preimy\cap T) = \Hs^\gamma(\preimy)>0.$$
Let $X=T\cap f^{-1}(Y)$. 

We now start from Federer \cite{Federer}, section
2.10.25.  The theorem there states:

\begin{theorem}\label{T:2}
If $f:X\to Y$ is a Lipschitzian map of metric
spaces, $A\subset X$, $0\le k<\infty$ and $0\le m<\infty$,
then
$$ \int_Y^* \Hs^k\left( A\cap\preimy \right) d\Hs^m y
\le (\Lip f)^m 
\frac{\alpha(k)\alpha(m)}{\alpha(k+m)}
\Hs^{k+m}(A), $$
provided either $\{y: \Hs^k(A\cap\preimy)>0\}$ is
the union of a countable family of sets with finite
$\Hs^m$ measure, or $Y$ is boundedly compact.
\qed
\end{theorem}

Here $\int^*$ denotes the upper integral,
$\Lip f$ is the Lipschitz constant of $f$,
and $\alpha(k)$ is the constant
$$ \alpha(k):= \frac{\Gamma(\frac12)^k}{\Gamma(\frac k2 +1)}, $$
where $\Gamma$ denotes the Euler Gamma function.

\ignore{
\subsection{Parenthetical Remark:} When working through this material 
in Federer's class
in 1969-70, I made a marginal note of his statement
that the proviso beginning ``provided$\ldots$" 
in the statement of Theorem \ref{T:2}  is redundant
and may be removed (as may the similar assumption
in his 2.10.22 and 2.10.24). Presumably this is
published somewhere, possibly even in a later
edition, and it is desirable to track it down or otherwise
prove it. However, we shall just use it
(or more exactly, use its proof) as it stands. 
}

\subsection{}
In the proof of Theorem \ref{T:2} (on page 189),
Federer used
the hypothesis that $f$ is Lipschitzian in order to
bound the diameter of the image $f(S)$ of a small
set $S$ by a constant multiple of $\diam(S)$.
If instead we can bound $\diam f(S)$ (for, say,  
$\diam(S)\le1$) by
a constant times $(\diam S)^s$, for some $s>1$,
then we can improve
the inequality in Theorem \ref{T:2}
to
$$ \int_Y^* \Hs^k\left( A\cap\preimy \right) d\Hs^m y
\le \textup{const}\cdot
\Hs^{k+sm}(A). $$
(Just replace $m$ by $sm$ at each occurrence
in Federer's argument.)

We can apply this, with the $X$ and $Y$ chosen above,
and the replacements 
$$ k\to\beta, m\to\gamma, s\to k, A\to X$$
so that it reads
$$ \int_Y^* \Hs^\gamma\left( X\cap\preimy \right) d\Hs^\beta y
\le C
\Hs^{\gamma+k\beta}(X). $$
Thus
$$ \int_Y^* \Hs^\gamma\left( \preimy \right) d\Hs^\beta y
\le C
\Hs^{\gamma+k\beta}(X), $$
But $\gamma+k\beta>1$, and $X\subset\R$, so 
$$\int_Y^* \Hs^\gamma\left( \preimy \right) d\Hs^\beta y=0.$$
This is impossible, because the integrand is positive
(in fact $+\infty$) for all $y\in Y$, and
$\Hs^\beta(Y)>0$.

Thus the lemma is proven.
\end{proof}

\section{Examples of large $L(\alpha,f)$}
\label{S:3}

Fix $k\in\N$.

\subsection{}
In order to prove that the lower bound \eqref{E:3}
holds for each $\alpha$ strictly between $0$ and $1$,
it suffices to show by construction that
for each given $r$,$s$, and $\epsilon$, 
with $2\le r\in\N$, $2\le s\in\N$, and 
$$0<\epsilon< \frac1{rs-1},$$
there is a function $f\in C^k(E,\R)$
having
$$ L\left(
\frac{\log_s r}{1+\log_s\left(\frac{r}{1-\epsilon}\right)}
,f
\right)
\ge
\frac{1}{(k+\epsilon)\left\{
1+\log_s\left(\frac{r}{1-\epsilon}\right)
\right\}
}.
$$
To see this, observe first that the set of
all numbers 
$$
\alpha(r,s):=\displaystyle\frac{\log_sr}{1+\log_sr},
\ 
(2\le r\in\N, 2\le s\in\N)
$$ 
is dense in the interval $(0,1)$.  So
given any positive $\alpha<1$, and any $\eta>0$,
we can choose $r$ and $s$  so that
$$
\alpha+\eta > \alpha(r,s) > \alpha,
$$
and then choose $\epsilon>0$ so that
$$ \frac{1}{
(k+\epsilon)
\left\{
1+\log_s\left(\frac{r}{1-\epsilon}\right)
\right\}
} >
\frac1{k(1+\log_s r)} -\eta
=
\frac{1-\alpha(r,s)}k - \eta
,$$
and 
$$
\frac{\log_s r}{1+\log_s\left(\frac{r}{1-\epsilon}\right)}
>\alpha.
$$ 
Then for the function constructed we have
$$ L(\alpha,f) \ge L\left(
\frac{\log_s r}{1+\log_s\left(\frac{r}{1-\epsilon}\right)}
,f
\right)
>
\frac{1-\alpha-\eta}k -\eta. $$
Since this can be done for
each positive $\eta$, we conclude that
$$ \sup\{L(\alpha,f):f\in C^k(E,\R)\} \ge
\frac{1-\alpha}k,$$
as stated.

\subsection{}
Fix $r$,$s$ and $\epsilon$ with
$2\le r\in\N$, $2\le s\in\N$, $0<\epsilon<\frac12$.

The $C^k$ function $f$ we are going to construct
will map the interval $E=[-2,2]$ to $\R$,
will be $k$-flat on a certain Cantor-type set
$A\subset[0,1]$, and will be strictly increasing
 when restricted to
$[-2,0]$ and when restricted to $[1,2]$.
All the action will occur in $[0,1]$.
The only critical points of $f$ will be
the points of $A$.
The function will map $[0,1]$ onto itself,
with $f(0)=0$ and $f(1)=1$.
The image $f(A)\subset[0,1]$ will be another
Cantor-type set $D$.  For
each $y\in D$, the preimage $\preimy$
will consist of a Cantor-type subset $A_y$ of $A$,
together with a set of isolated points.
Each set $A_y$ will have the same Hausdorff
dimension 
$$
\gamma(r,s,\epsilon):= 
\frac{\log_s r}{1+\log_s\left(\frac{r}{1-\epsilon}\right)}
.
$$
Thus
$$I_{
\gamma(r,s,\epsilon)
}(f)=D,$$ and hence
$L(
\gamma(r,s,\epsilon)
,f) = \dim D$.
The Hausdorff dimension of $D$ will be
$$\frac{1}{(k+\epsilon)\left\{
1+\log_s\left(\frac{r}{1-\epsilon}\right)
\right\}
}.
$$

\medskip
The construction involves the 
repetition at reducing scales of a single pattern.
It involves generations of nested rectangles. 
The graph $f$ traverses
each rectangle from bottom left to top right.
The first generation has the single rectangle
$[0,1]\times[0,1]$. Each rectangle of the
$n$-th generation contains $rs$ rectangles of
the next generation, arranged in $s$ rows of
$r$ congruent rectangles. The rectangles of
different generations are not similar, i.e.
they do not have the same proportions of height
to width.  In other words, the rescaling from
generation to generation is non-isotropic, non-conformal.
In fact the ratio of height to width
tends rapidly to zero as we progress through the
generations. The rectangles of a given generation
project vertically to pairwise-disjoint closed
intervals, but their horizontal projections 
overlap. More precisely, the horizontal projections
of each given row are identical, and the horizontal
projections of distinct rows are pairwise-disjoint.

\begin{center}
{%
\setlength{\fboxsep}{0pt}
\setlength{\fboxrule}{1pt}
\fbox{%
\includegraphics[width=0.9\textwidth]{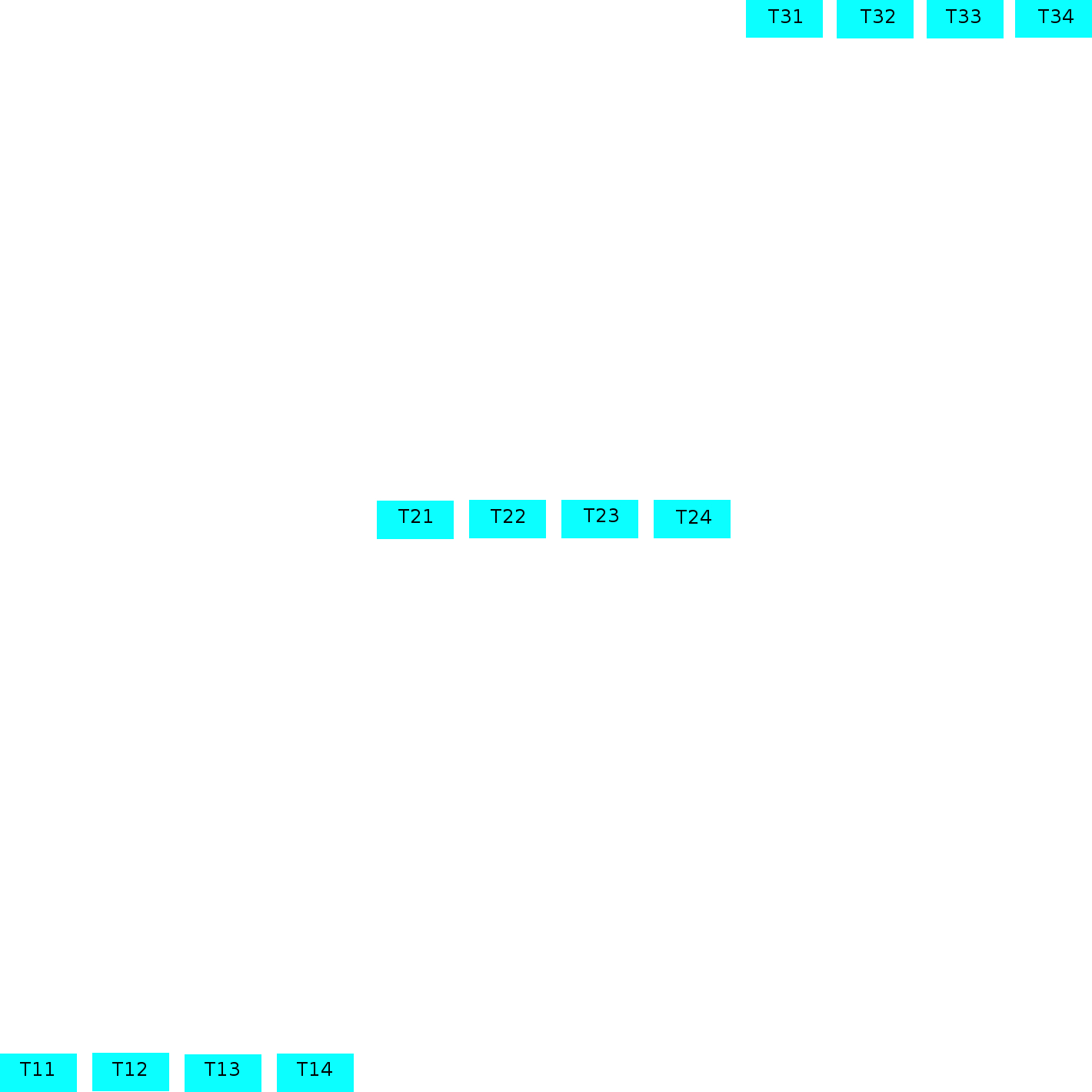}
\hskip-4pt
}%
}%
\\ Figure 1
\end{center}

Figure 1 shows the pattern of the 12 rectangles of
generation 1 for the case when $r=4$ and $s=3$.

We denote the vertical projection of a rectangle
$T$ by $T_1$, and its horizontal projection by
$T_2$. 

We use the following notation for features
of the $n$-th generation:
\begin{description}
\item[$R_n$] the family of all its rectangles,
\item[$c_n$] width of each rectangle, which is
to depend only on $n$,
\item[$a_n$] height of each rectangle, also depending
only on $n$,
\item[$P_n$] the family of all vertical projections 
$T_1$ of $T\in R_n$ (on $\R$),
\item[$Q_n$] the family of horizontal projections 
$T_2$ of $T\in R_n$,
\end{description}

If $n>1$ and $T\in R_{n-1}$, then we use the following
notation for features of the $n$-generation rectangles
contained in $T$:
\begin{description}
\item[$R_T$] the family of all these rectangles,
\item[$R_{Tm}$] the $m$-th row of $R_T$,
\item[$R_{Tmp}$] the $p$-th rectangle of $R_{Tm}$,
\item[$P_T$] the family of vertical projections
of $R_{T}$,
\item[$Q_T$] the family of horizontal projections of $R_T$,
\item[$d_n$] horizontal space between adjacent rectangles 
$P_{T}$, also depending only on $n$,
\item[$b_n$] vertical spacing between adjacent intervals of $Q_n$,
also depending only on $n$.
\end{description}

In Figure 1, the second-generation rectangle 
$R_{Tmp}$, where $T=[0,1]\times[0,1]$, is labelled
$Tmp$. 

The first rows $R_{T1}$ of rectangles of the
$n$-th generation are placed at the bottom
left of their containing rectangle $T\in R_{n-1}$,
and the last rows $R_{Ts}$ are placed at
the top right. The remaining rows are evenly
spaced vertically in $T$ (the space is $b_n$), 
and the horizontal gap between rows has the same
width $d_n$ 
as the horizontal gap between the rectangles
of each row $R_{Tm}$.

With this notation, we can define the Cantor-type sets
$A$ and $D$ as
$$ A:= \bigcap_{n=1}^\infty \bigcup P_n,
\quad 
D:= \bigcap_{n=1}^\infty \bigcup Q_n.
$$

For $n>1$, and $T\in R_{n-1}$,
the function $f$ is defined on the union of gaps
$T_1\setminus P_T$ in such a way that its
graph links the top right vertex of each $n$-th
generation rectangle
$R_{Tmp}$ (apart from the last one, $P_{Tsr}$)
to the bottom left vertex of the next to its right.
(If $p<r$, the next rectangle is $R_{Tm(p+1)}$,
whereas if $p=r$, the next rectangle is
$R_{T(m+1)1}$.) This linking is illustrated in
Figures 2 and 3. Figure 2 shows the link between
the two rectangles of the same row (specifically,
the first two rectangles of the first row of the
next generation inside rectangle $T$).
Figure 3 shows the link between the last rectangle 
of one row and the first rectangle of the next row.
 
\begin{center}
\includegraphics[width=\textwidth]{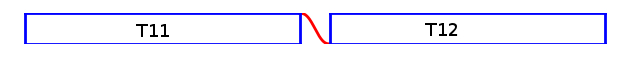}
\\Figure 2.
\end{center}
\begin{center}
\includegraphics[width=\textwidth]{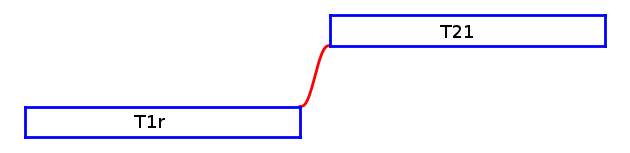}
\\Figure 3.
\end{center}

Each of these linking graphs
is defined by a formula of the form
$$ f(x):= \delta\cdot\phi\left(\frac{x-\mu}{\nu}\right)+\rho, $$
where the parameters $\delta$, $\mu$, $\nu$ and $\rho$
depend on $T$, $m$ and $p$, but the function
$\phi$ is always the same. 
To be more specific, $\phi:[0,1]\to[0,1]$
is a $C^k$ function that is flat of order $k$ at $0$ and $1$,
and decreases strictly monotonically from $1$ to
$0$.  A specific example of such a function is
given by
$$ \phi(x):= 1 - 
\frac{\int_0^x t^{k}(1-t)^{k}\,dt}{\int_0^1 t^{k}(1-t)^{k}\,dt}. 
$$
The parameters are chosen so that the top right
of $R_{Tmp}$ has coordinates $(\rho,\mu)$,
and the bottom left of the next rectangle
has coordinates $(\rho+\nu,\mu+\delta)$.
Thus $\delta=-a_n$ when $R_{Tmp}$ is not the last
rectangle of its row $R_{Tm}$, 
and $\delta=b_n$ when $R_{Tmp}$ is last
in its row, i.e when $m<s$ and $p=r$.
The parameter $\nu$, the width of the horizontal
gap between successive rectangles of the same generation,
is $d_n$, depending only on $n$.

If we set
$$ K:= \sup_{[0,1]} \phi^{(k)}, $$
then in the gap after $R_{Tmp}$ we have
$$ \sup |f^{(k)}| \le \frac{K\delta}{\nu^k}
\le \frac{K\max(a_n,b_n)}{d_n^k} < \frac{Ka_{n-1}}{d_n^k}.
$$
Thus as long as we ensure that $a_{n-1}/d_n^k\to0$ as 
$n\uparrow\infty$, we will end up with a $C^k$
function $f$, flat of order $k$ on $A$. 

(If this seems puzzling, note that at later generations
the horizontal gaps between rectangles will be far larger than
the height of the rectangles and even much larger
than the vertical gaps between rows of rectangles
than they are in the figures. The rectangles become very
very low, compared with their width.)

We now specify the specific values that determine
the function $f$ completely.  We take $a_1=c_1=1$, 
and then for $n>1$ we define
$$
 c_n :=  \frac{(1-\epsilon)c_{n-1}}{rs} 
= \left(\frac{1-\epsilon}{rs}\right)^{n-1}, 
$$
so that the intervals of each $P_T$ take up
all but a proportion $\epsilon$ of the
projection $T_1$.   The horizontal gap-length $d_n$
is then forced to be
$$ d_n = \frac{\epsilon c_{n-1}}{rs-1} 
> \frac{ \epsilon(1-\epsilon)^{n-2} }{(rs)^{n-1}}. $$
Next, we take
$$  a_n:= d_n^{k+\epsilon}, $$
and this ensures that $a_{n-1}/d_n^k\to0$, as required.
The value of the vertical gap-height $b_n$ is then forced to be
$$ b_n = \frac{a_{n-1} - sa_n}{s-1}. $$
 
The Cantor set $D$ is covered by the $s^n$ disjoint
horizontal projections of the family $R_{n+1}$, each of
length $a_{n+1}$, and for each $n$ we have
$$ s^n a_{n+1}^\beta = s^nd_{n+1}^{(k+\epsilon)\beta}
=
\left( 
\frac{\epsilon rs} {(1-\epsilon)(rs-1)}
\right)^{(k+\epsilon)\beta}
$$
(independently of $n$)
when
$$\beta = 
\frac{1}{
(k+\epsilon)
\left\{
1+\log_s\left(\frac{r}{1-\epsilon}\right)
\right\}
}
,$$
and we conclude that this value of $\beta$ is 
the Hausdorff dimension of $D$ (Compare Example
4.4, page 69, in Falconer's text \cite{Falconer},
of which our set $D$ is a special case).

\medskip
Now for $y\in D$, we consider $\preimy$. 
The preimage of $y$
consists of a countable subset of $[0,1]\setminus A$,
together with the Cantor set $A_y$, and this Cantor set
may be described as follows.
The rectangle $[0,1]\times[0,1]$ contributes the
interval $[0,1]$ as the first approximation to $A_y$.
A rectangle
$T\in R_n$ contributes an interval $T_1$ to the
$n$-th approximation to $A_y$ if and only if
$y\in T_2$. 
The set $A_y$
is the intersection of these approximations, and
each approximation covers $A_y$ by $r^n$ intervals
of length $c_n$.  These intervals may be
shrunk by a certain amount, and still cover,
because for each $T$, 
$$ \diam(A_y\cap T_1) < \diam(T_1). $$
In fact, $\inf A_y\cap T_1$ is the left-hand end
of the projection of the first rectangle of
some row of the next generation inside $T$,
whereas $\sup A_y\cap T_1$ lies somewhere inside
the projection of the last rectangle of the same
row.
However the ratio
$$ \lambda(r,s,\epsilon)
 := \frac{\diam(A_y\cap T_1)}{\diam(T_1)} $$
depends only on $r$,$s$ and $\epsilon$,
(and lies between $(r-1)/rs$ and $1/s$),
so $A_y$ is covered by $r^{n-1}$ intervals
of length $\lambda c_n$.  (In fact, a calculation
shows that, perhaps surprisingly, the ratio 
$$\lambda = \frac{r-1}{rs-1},$$  
does not depend on $\epsilon$.) 

For each $n\in\N$ we have that
$$r^{n-1} (\lambda c_{n})^\alpha = \lambda^\alpha
\left(
\frac{(1-\epsilon)^\alpha r^{1-\alpha}}{s^\alpha}
\right)^{n-1} = \lambda^\alpha,
$$ 
when
$$\alpha = \frac{\log_s r}{1+\log_s\left(\frac{r}{1-\epsilon}\right)},$$
and we conclude that 
this value of $\alpha$ is the
Hausdorff dimension of $A_y$ (--- the set $A_y$
is another case of 
Example 4.4 in \cite{Falconer}).

This completes the contruction, and the proof.

\subsection{Remarks.}
1. For $f\in C^k(E,\R)$, we have
$$ 0\le L(1,f) \le \inf_{0<\alpha<1} L(\alpha,f) = 0 $$
so the equality in Theorem \ref{T:1} also holds
when $\alpha=1$.

2. The construction we gave shows that the Hausdorff
dimension of the image of the critical set of
a $C^k$ function on $\R$ may be arbitrarily close
to $1/k$, i.e. that the bound mentioned in
Subsection \ref{SS:1-1} is sharp.

3. At the other extreme,
the construction also shows that the dimension of each one
of an uncountable
family of level sets may be arbitrarily
close to $1$.  

4. By tweaking the construction, keeping $r=2$
but increasing $s$ at each step, one obtains a
$C^k$ function having a critical set $A$ of
dimension zero, with $\dim f(A)=1/k$.

5. By, instead, keeping $s=2$ but increasing
$r$ at each step, one obtains a $C^k$ function
having an uncountable family of level
sets of Hausdorff dimension $1$.

\subsection{Acknowledgement}
We are grateful to Ian Short for valuable comments on a draft of this paper,
and to the referee for suggestions that improved the exposition. We also thank
M\'arton Elekes for pointing out two errata in the published version of the historical account in Subsection 1.5, second paragraph.  We have corrected these in the present version.

\end{document}